\documentclass[11pt]{amsart}
\usepackage{graphicx}
\usepackage{float}
\usepackage{color}

\input{amssymb.sty}

\newtheorem{theorem}{Theorem}

\newtheorem{lemma}[theorem]{Lemma}
\newtheorem{definition}[theorem]{Definition}

\date{\today}
\begin{document}

\title[On the Dimension of Matrix Representations of Nilpotent Groups]{On the  Dimension of Matrix Representations of Finitely Generated Torsion Free Nilpotent Groups}

\author[M. Habeeb]{Maggie Habeeb}
\address{California University of Pennsylvania}%
\email{Habeeb@CalU.edu}
\thanks{Research of M.H. was supported partially by the NSF-LSAMP fellowship.}
\author[D. Kahrobaei]{Delaram Kahrobaei}
\address{CUNY Graduate Center and City Tech, City University of New York}%
\email{DKahrobaei@GC.Cuny.edu}
\thanks{Research D.K. was partially supported by the Office of Naval Research grant N000141210758, PSC-CUNY grant from the CUNY research foundation, as well as the City Tech foundation.}

\maketitle
\begin{abstract}It is well known that any polycyclic group, and hence any finitely generated nilpotent group, can be embedded into $GL_{n}(\mathbb{Z})$ for an appropriate $n\in \mathbb{N}$; that is, each element in the group has a unique matrix representation.  An algorithm to determine this embedding was presented in \cite{nickel}. In this paper, we determine the complexity of the crux of the algorithm and the dimension of the matrices produced as well as provide a modification of the algorithm presented in \cite{nickel}. \end{abstract}

\section{Background Information} In this section we will review basic facts about polycyclic and nilpotent groups.  We refer the reader to \cite{holt} and \cite{PH57} for more information on these groups.
\subsection{Polycyclic Groups}
\begin{definition}\cite{holt}
A group is called polycyclic if it admits a finite subnormal series
\begin{eqnarray*}
G=G_{1} \triangleright G_{2} \triangleright G_{3} \triangleright \cdots \triangleright G_{n+1}=1
\end{eqnarray*}
where each $G_{i}/G_{i+1}$ is cyclic.\end{definition}
  The number of infinite factors in the polycyclic series is called the \textit{Hirsch length}, and is independent of the polycyclic series chosen. Since each factor is cyclic there exists an $x_{i}\in G$ such that $\left\langle x_{i}G_{i+1}\right\rangle=G_{i}/G_{i+1}$.  We call the sequence $X=[x_{1}, x_{2}, \cdots, x_{n}]$ a polycyclic sequence for $G$.  The sequence of relative orders of $X$ is the sequence $R(X)=(r_{1}, \cdots, r_{n})$ where $r_{i}=\left[G_{i}:G_{i+1}\right]$.  We denote the set of indices in which $r_{i}$ is finite by $I(X)$. \\
\indent Polycyclic groups have finite presentation
\begin{eqnarray*} \left\langle a_{1}, \cdots, a_{n}; a_{i}^{a_{j}}=w_{ij},  a_{i}^{a_{j}^{-1}}=v_{ij},  a_{k}^{r_{k}}=u_{kk} \text{ for } k\in I, 1\leq j< i \leq n \right\rangle\end{eqnarray*} where $r_{i}\in \mathbb{N} \cup \infty$, $r_{i}< \infty$ if $i\in I\subseteq\left\{1, 2,\cdots, n\right\}$ and $w_{ij}, v_{ij}, u_{jj}$ are words in the generators $a_{j+1}, \cdots, a_{n}$. The relations $a_{i}^{a_{j}}=w_{ij},  a_{i}^{a_{j}^{-1}}=v_{ij}$ are called conjugacy relations, while the relations $a_{k}^{r_{k}}=u_{kk}$ are called power relations. If $r_{i}=[G_{i}:G_{i+1}]$ for each $i\in \left\{1, 2,\cdots, n\right\}$ then this presentation is called a consistent polycyclic presentation.  Every polycyclic group admits a consistent polycyclic presentation, which results in the following normal form.
\begin{definition}\cite{holt} Let $X=\left[x_{1}, \cdots, x_{n}\right]$ be a polycyclic sequence for $G$ and $R(X)=(r_{1}, \cdots, r_{n})$ be its sequence of relative orders.  Then every $g\in G$ can be written uniquely in the form $g=x_{1}^{e_{1}}\cdots x_{n}^{e_{n}}$ with  $e_{i}\in \mathbb{Z}$ and $0 \leq e_{i} < r_{i}$.  This expression is called the normal form of $G$ with respect to $X$.\end{definition}
\indent The normal form for an element in a group given by a consistent polycyclic presentation can be determined via the \textit{collection algorithm}. Thus, the collection algorithm gives a solution to the word problem for a group $G$ given by a consistent polycyclic presentation.  The collection algorithm works by iteratively applying the power and conjugacy relations of the presentation until the normal form for a word is obtained.  The nature of the power and conjugacy relations ensure that the collection algorithm terminates (see \cite{holt}). The collection algorithm is currently the best known way to obtain the normal form of an element in a polycyclic group, although the worst case time complexity is unknown. \\
\indent By using different representations of elements in a polycyclic group one can solve some group theoretic problems such as the  conjugacy search problem. This can be done by using a different representation of group elements. For example to find a different representation of group elements one can use the well known fact, due to L. Auslander (see \cite{kos}), that every polycylic group is linear; that is, every polycyclic group $G$ can be embedded into $GL_{n}(\mathbb{Z})$ for an appropriate $n$.  \\
\subsection{Nilpotent Groups}
\begin{definition}\cite{PH57} A group, $G$, is called nilpotent of class $c\geq 1$ if  \begin{eqnarray*}[y_{1}, y_{2}, \cdots, y_{c+1}]=1 \text{ for any } y_{1},y_{2}, \cdots, y_{c+1} \in G,\end{eqnarray*} where $[y_{1},y_{2},y_{3}]=[[y_{1},y_{2}],y_{3}]$.  Equivalently we may define a nilpotent group as follows. We define a lower central series of $G$ inductively: let $\gamma_{1}(G)=G$, $\gamma_{2}(G)=[G,G]$, and $\gamma_{k}(G)=[\gamma_{k-1}(G),G]$. If $\gamma_{k}(G)=\left\{e\right\}$ for some $k$, then $G$ is nilpotent. \end{definition}

Finitely generated nilpotent groups are in fact polycyclic.  Let $G$ be a finitely generated torsion-free nilpotent group. Since $G$ is torsion-free and finitely generated nilpotent, there exists a central series
 \begin{eqnarray*}
 G=G_{1}\triangleright G_{2} \triangleright G_{3} \cdots \triangleright G_{n+1}=\left\{1\right\}
 \end{eqnarray*}
 such that for each $1\leq r \leq n$ the factor $G_{r}/G_{r+1}$ is infinite cyclic generated by $x_{r}G_{r+1}$ for some $x_{r}\in G_{r}$.  Given such $x_{1}, \cdots, x_{n}$ each element $x\in G$ has unique normal form
\begin{eqnarray*}
x=x_{1}^{e_{1}}\cdots x_{n}^{e_{n}},
\end{eqnarray*}
where $e_{i}\in \mathbb{Z}$.  The $n$-tuple $(e_{1}, \cdots, e_{n})$ is called the vector of exponents of $x$.\\
\indent As mentioned above all finitely generated nilpotent groups are polycyclic, and hence admit a polycyclic presentation. In particular every finitely generated nilpotent group admits a presentation with conjugacy relations of the form:
\begin{eqnarray*}
x_{i}^{x_{j}}=x_{i}x_{i+1}^{b_{i,j,i+1}}\cdots x_{n}^{b_{i,j,n}} \text{ for } 1\leq j<i\leq n,
\end{eqnarray*}
\begin{eqnarray*}
x_{i}^{x_{j}^{-1}}=x_{i}x_{i+1}^{c_{i,j,i+1}}\cdots x_{n}^{c_{i,j,n}} \text{ for } 1\leq j<i\leq n,
\end{eqnarray*}
where the $c_{i,j,k}, b_{i,j,k} \in \mathbb{Z}$.
A polycyclic presentation with these conjugacy relations is called a \textit{nilpotent presentation}.  Since $G$ is torsion free, there exists a consistent nilpotent presentation with each $r_{i}=\infty$. Hence, for any finitely generated torsion-free nilpotent group we may find a consistent nilpotent presentation of the form:\\
\resizebox{\linewidth}{!}{$\left\langle x_{1}, \cdots, x_{n}; x_{i}^{x_{j}}=x_{i}x_{i+1}^{b_{i,j,i+1}}\cdots x_{n}^{b_{i,j,n}}, x_{i}^{x_{j}^{-1}}=x_{i}x_{i+1}^{c_{i,j,i+1}}\cdots x_{n}^{c_{i,j,n}} \text{ for } 1\leq j<i\leq n\right\rangle.$}\\
 Let $G$ be a finitely generated torsion free nilpotent group with a presentation as above. Then each $x\in G$ may be written uniquely in the form $x=x_{1}^{e_{1}}\cdots x_{n}^{e_{n}}$. Given two elements $x=x_{1}^{e_{1}}\cdots x_{n}^{e_{n}}, y=x_{1}^{y_{1}}\cdots x_{n}^{y_{n}}$, their product can be written uniquely as:
\begin{eqnarray*}
xy=x_{1}^{e_{1}}\cdots x_{n}^{e_{n}}x_{1}^{y_{1}}\cdots x_{n}^{y_{n}}=x_{1}^{z_{1}}\cdots x_{n}^{z_{n}}.
\end{eqnarray*}
Philip Hall \cite{PH57} showed that there are rational polynomials $f_{1}, \cdots, f_{n}$ that describe the multiplication of elements in $G$, a finitely generated torsion free nilpotent group. In particular, if $(e_{1}, \cdots, e_{n})$, $(y_{1}, \cdots, y_{n})$, and $(z_{1}, \cdots, z_{n})$ are the vector of exponents for $x, y, \text{ and } xy$, respectively, then for $1 \leq r \leq n$ we have
\begin{eqnarray*}
z_{r}=f_{r}(e_{1}, \cdots, e_{n},y_{1}, \cdots, y_{n}).
\end{eqnarray*}  In \cite{LGS} Leedham-Green and Soicher show how to compute polynomials in variables corresponding to the $e_{i}, y_{i}, c_{i,j,k}$ such that they describe the multiplication of elements in the group $G$ by an algorithm referred to as ``Deep Thought".  In \cite{nickel} W. Nickel  presents an algorithm for computing a presentation of a finitely generated torsion free nilpotent group given by a polycyclic presentation by unitriangular matrices over $\mathbb{Z}$.  The algorithm by Nickel uses the polynomials computed by ``Deep Thought."  There is another algorithm due to DeGraaf and Nickel that computes a faithful unitriangular representation of finitely generated torsion-free nilpotent groups in \cite{degraaf}, but the algorithm presented by Nickel in \cite{nickel} is more efficient.
\section{Matrix Representations of Finitely Generated Torsion-free Nilpotent Groups}
W. Nickel introduced an algorithm to compute a matrix representation for a finitely generated torsion-free nilpotent group given by a nilpotent presentation in \cite{nickel}.  The crux of the algorithm computes a $\mathbb{Q}$-basis for a finite dimensional faithful $G$-module, where $G$ is a finitely generated torsion-free nilpotent group given by a nilpotent presentation.\\
\indent In \cite{nickel}, a faithful finite dimensional $G$-module is constructed as follows.  Let $G$ be a finitely generated torsion-free nilpotent group with nilpotent generating sequence $a_{1}, \cdots, a_{n}$ and multiplication polynomials $q_{1}, \cdots,q_{n}$. Let $G$ act on the dual of the group ring $(\mathbb{Q}G)^{*}$ by defining $f^{g}(h):=f(hg^{-1})$. By identifying $a_{1}^{x_{1}}\cdots a_{n}^{x_{n}}$ with $x_{1}, \cdots,x_{n}$ the image of $f\in (\mathbb{Q}G)^{*}$ under the action of $G$ can be described using the multiplication polynomials $q_{1}, \cdots, q_{n}$; that is, $f^{g}=f(q_{1}, \cdots,q_{n})$. With this action one can construct a finite dimensional faithful $G$-submodule of $(\mathbb{Q}G)^{*}$ by using the following lemma.
\begin{lemma} \cite{nickel} The submodule $M$ of $(\mathbb{Q}G)^{*}$ generated by $t_{i}:G \rightarrow \mathbb{Z}$ with $t_{i}(a_{1}^{x_{1}}\cdots a_{n}^{x_{n}})=x_{i}$ is a finite dimensional faithful $G$-module.
\end{lemma}
\indent To construct a basis a method called \textit{Insert} is used.  \textit{Insert} adds a polynomial to a given basis of polynomials if that polynomial is not in the span of the basis; that is, if the polynomial cannot be written as a linear combination of the basis elements.  This is done by first placing an ordering on the monomials.  The leading monomial of a polynomial is then the monomial that is largest with respect to the ordering.  \textit{Insert} takes a basis of polynomials in ascending order and a polynomial $f$, and determines if $f$ is in the span of the basis. If $f$ is not in the span of the basis, then it will be added to the basis. \textit{Insert} determines if $f$ is in the span of the basis by subtracting the appropriate multiple of each polynomial in the basis, starting with the polynomial that is largest with respect to the ordering and continuing in descending order. If $f\neq 0$ after this process is completed, then $f$ is not in the span of the basis and is added to the original basis.\\
 \indent In order to determine a $\mathbb{Q}$-basis for a finite dimensional $G$-module generated by $f_{1}, \cdots, f_{k}$, the algorithm uses \textit{Insert} to build up a basis from $f_{1}, \cdots, f_{k}$.  Then the algorithm works up the central series of $G$ beginning with $G_{n}$.  For each $j$ the algorithm will compute a basis for the $G_{j}$-module $M_{j}$ from a basis of the $G_{j+1}$-module $M_{j+1}$ generated by $f_{1}, \cdots, f_{k}$. In order to obtain a generating set for $M_{j}$ it is enough to close the module $M_{j+1}$ under the action of powers of $a_{j}$ (since $G_{j}/G_{j+1}$ is cyclic) and to apply powers of $a_{j}$ to each basis element of $M_{j+1}$. The algorithm for computing a $\mathbb{Q}$-basis is presented in Figure 1.\\
  \indent  The matrix representation for each generator $a_{k}$ of $G$ can be calculated by decomposing the image of each basis element under $a_{k}$ in terms of the basis.  The ordering utilized in \textit{Insert} is the reverse lexicographic ordering; that is, $x_{1}^{k_{1}}\cdots x_{n}^{k_{n}}<x_{1}^{l_{1}}\cdots x_{n}^{l_{n}}$ if there is an $i$ ($1\leq i \leq n$) such that $k_{j}=l_{j}$ for $i<j\leq n$ and $k_{i}<l_{i}$. Suppose that $b_{1}, \cdots, b_{m}$ is the basis produced using this ordering. Let $q_{1}^{(j)}, \cdots q_{n}^{(j)}$ be the polynomials that describe the multiplication of an arbitrary element of the group with $a_{j}^{-1}$:
      \begin{eqnarray*}
      a_{1}^{x_{1}}\cdots a_{n}^{x_{n}}\cdot a_{j}^{-1}=a_{1}^{q_{1}^{(j)}} \cdots a_{n}^{q_{n}^{(j)}}
      \end{eqnarray*}
 The polynomials $q_{i}^{(j)}$ are obtained from the multiplication polynomials $q_{i}$ by setting $y_{i}=0$ if $i \neq j$ and $y_{i}=-1$ if $i=j$.     Then for each basis element $b_{k}$, one can compute $b_{k}(q_{1}^{(j)}, \cdots, q_{n}^{(j)})$ and decompose this element into a linear combination $c_{k_{1}}b_{1}+\cdots + c_{k_{n}}b_{m}$ of basis vectors.  The coefficients $c_{k_{1}}, \cdots, c_{k_{n}}$ are the $k^{th}$ row of the matrix representing $a_{j}$.  For more information on this algorithm see \cite{nickel}.\\

\begin{small}
\begin{figure}[t]
\label{Matrix Representation Algorithm (building a basis)}
\caption{Matrix Representation Algorithm (building a basis), \cite{nickel}}
\flushleft \textbf{Input:} A finitely generated torsion- free nilpotent group $G$ with a nilpotent generating sequence $a_{1}, \cdots, a_{n}$ and corresponding multiplication polynomials $q_{1}, \cdots, q_{n}$; a list of polynomials $f_{1}, \cdots, f_{k}$.\\
\flushleft \textbf{Output:} A $\mathbb{Q}$-basis $B$ for the $G$-module generated by $f_{1}, \cdots, f_{k}$.\\
\flushleft \textbf{Algorithm:}\\
$B:=[]$;\\

 \textbf{for} $j$ \textbf{in} $[1 \cdots k]$ \textbf{do} Insert($B, f_{j})$; \textbf{od};\\
 \flushleft \textbf{for} $j$ \textbf{in} $[n, n-1\cdots 1]$ \textbf{do}\\
 \qquad \#B is a basis for $M_{j+1}$\\
 \qquad \# Exponents after multiplication by $a_{i}^{-1}$ from the right:\\
 \qquad $q^{(j)}:=[q_{1}(x_{1}, \cdots, x_{n}, y_{i}=-\delta_{ij}),\cdots, q_{n}(x_{1}, \cdots, x_{n}, y_{i}=-\delta_{ij})]$;\\
 \qquad \textbf{for} $f$ \textbf{in} Copy($B$) \textbf{do}\\
 \quad\qquad \# Add to $B$ images of $f$ under powers of $a_{j}$\\
 \quad \qquad \textbf{repeat} \\
 \quad \qquad \qquad  \# Compute $f^{a_{j}}$\\
 \quad \qquad \qquad $f^{a_{j}}:=f(q_{1}^{(j)}, \cdots, q_{n}^{(j)})$;\\
 \quad \qquad \qquad r:=Insert($B,f^{a_{j}}$);\\
 \quad \qquad \qquad $f:=f^{a_{j}}$;\\
 \quad \qquad \textbf{until} r=0;\\
 \qquad \textbf{od};\\
 \quad \textbf{od};\\
 \flushleft \textbf{return} $B$;

\end{figure}
\end{small}
\section{Complexity Analysis}
The algorithm for finding matrix representations of torsion-free finitely generated nilpotent groups formulated in \cite{nickel} was implemented in the \textbf{GAP} package polycyclic.  The algorithm was implemented using the coordinate functions $t_{i}: G\rightarrow \mathbb{Z}$ given by $a_{1}^{x_{1}}\cdots a_{n}^{x_{n}} \mapsto x_{i}$ for $1 \leq i \leq n$ as the input polynomials $f_{1}, \cdots, f_{k}$. In order to analyze the complexity of the algorithm utilized for building a $\mathbb{Q}$-basis (see Figure 1), we must analyze how the action of $a_{j}\in G$ affects each coordinate function $t_{i}$.  Recall that $t_{i}^{a_{j}}:=t_{i}(q_{1}^{(j)}, \cdots, q_{n}^{(j)})$ where $q_{i}^{(j)}=q_{i}(x_{1}, \cdots, x_{n}, y_{i}=-\delta_{ij})$ and the $q_{i}$ are the multiplication polynomials computed via ``Deep Thought". From the nature of the polynomials $q_{i}$ it follows that $q_{i}^{(j)}=x_{i}$ for $1 \leq i <j$, $q_{j}^{(j)}=x_{j}-1$, and $q_{i}^{(j)}=x_{i}+\overline{q_{i}}^{(j)}(x_{1}, \cdots, x_{i-1})$ for $ j <i \leq n$.  To see this, we will follow the exposition of \cite{nickel}.  We may rewrite the product, $a_{1}^{x_{1}}\cdots a_{n}^{x_{n}}a_{1}^{y_{1}}\cdots a_{n}^{y_{n}}$, of two elements $x=a_{1}^{x_{1}}\cdots a_{n}^{x_{n}}$, $y=a_{1}^{y_{1}}\cdots a_{n}^{y_{n}}\in G$  as \begin{eqnarray*} (a_{1}^{x_{1}}\cdots a_{i}^{x_{i}})(a_{i+1}^{x_{i+1}}\cdots a_{n}^{x_{n}}a_{1}^{y_{1}}\cdots a_{i}^{y_{i}})(a_{i+1}^{y_{i+1}}\cdots a_{n}^{y_{n}}),\end{eqnarray*} which is equal to 
\begin{eqnarray*} (a_{1}^{x_{1}}\cdots a_{i}^{x_{i}}a_{1}^{y_{1}}\cdots a_{i}^{y_{i}})(a_{i+1}^{x^{'}_{i+1}}\cdots a_{n}^{x^{'}_{n}}a_{i+1}^{y_{i+1}}\cdots a_{n}^{y_{n}})
\end{eqnarray*}
since $G_{i+1}$ is normal in $G$. The expression $a_{i+1}^{x^{'}_{i+1}}\cdots a_{n}^{x^{'}_{n}}a_{i+1}^{y_{i+1}}\cdots a_{n}^{y_{n}}$ can be computed in $G_{i+1}$ and does not involve $a_{1}\cdots a_{i}$. 
Hence, $q_{i}$ is determined by $a_{1}^{x_{1}}\cdots a_{i}^{x_{i}}a_{1}^{y_{1}}\cdots a_{i}^{y_{i}}$ and $q_{i}$ only depends on $x_{1}, \cdots, x_{i}, y_{1}\cdots, y_{i}$. Since $a_{i}$ is central in $G/G_{i+1}$ we have that $q_{i}= x_{i}+y_{i}+\overline{q_{i}}$ with $\overline{q_{i}}\in \mathbb{Q}[x_{1}, \cdots, x_{i-1},y_{1}, \cdots, y_{i-1}]$. Since multiplying $a_{1}^{x_{1}}\cdots a_{n}^{x_{n}}$ by $a_{j}^{-1}$ from the right does not affect $a_{1}, \cdots, a_{j-1}$ we have that $q_{i}^{(j)}=x_{i}$ for $1 \leq i <j$. Moreover, we have $q_{j}^{(j)}=x_{j}-1$ and by entering $y_{i}=-\delta_{ij}$ we have $q_{i}^{(j)}=x_{i}+\overline{q_{i}}^{(j)}(x_{1}, \cdots, x_{i-1})$ for $ j <i \leq n$ with $\overline{q_{i}}^{(j)}\in\mathbb{Q}[x_{1}, \cdots, x_{i-1}]$.

Here we would like to introduce some notation.  Since we are required to close the module under powers of each $a_{j}$ for $j=1, \cdots, n,$ we would like to describe the multiplication of an arbitrary group element with $a_{j}^{k}$ for $k\in \mathbb{N}$ :  \begin{eqnarray*}
      a_{1}^{x_{1}}\cdots a_{n}^{x_{n}}\cdot a_{j}^{-k}=a_{1}^{q_{1,k}^{(j)}} \cdots a_{n}^{q_{n,k}^{(j)}}.
      \end{eqnarray*}
 The polynomials $q_{i,k}^{(j)}$ are obtained from the multiplication polynomials $q_{i}$ by setting $y_{i}=0$ if $i \neq j$ and $y_{i}=-k$ if $i=j$, which we denote $-k\delta_{ij}$.   Note that by the same reasoning as above we have $q_{i,k}^{(j)}=x_{i}$ for $1 \leq i <j$, $q_{j,k}^{(j)}=x_{j}-k$, and $q_{i,k}^{(j)}=x_{i}+\overline{q_{i,k}}^{(j)}(x_{1}, \cdots, x_{i-1})$ for $ j <i \leq n$ with $\overline{q_{i,k}}^{(j)}\in\mathbb{Q}[x_{1}, \cdots, x_{i-1}]$.

 In order to understand the action of $a_{j}^{k}$, for $k\in \mathbb{N}$, on each $t_{i}$ we look at three cases: $1\leq i <j$, $i=j$, and $j < i \leq n$. 
\begin{small}
\begin{enumerate}
\item $1 \leq i <j$:
\begin{align*}t_{i}^{a_{j}^{k}}&:=t_{i}(q_{1,k}^{(j)}, \cdots, q_{n,k}^{(j)})\\&=t_{i}(x_{1}, \cdots, x_{j}-k, x_{j+1}+\overline{q_{j+1,k}}^{(j)}(x_{1}, \cdots, x_{j}), \cdots, x_{n}+\overline{q_{n,k}}^{(j)}(x_{1}, \cdots, x_{n-1})) \\&=x_{i}\\&=t_{i}\end{align*}

\item $i=j$:
\begin{align*}t_{j}^{a_{j}^{k}}&:=t_{j}(q_{1,k}^{(j)}, \cdots, q_{n,k}^{(j)})\\&=t_{i}(x_{1}, \cdots, x_{j}-k, x_{j+1}+\overline{q_{j+1,k}}^{(j)}(x_{1}, \cdots, x_{j}), \cdots, x_{n}+\overline{q_{n,k}}^{(j)}(x_{1}, \cdots, x_{n-1})) \\&=x_{j}-k\\&=t_{j}-k\end{align*}

\item $j<i \leq n$:
\begin{align*}t_{i}^{a_{j}^{k}}&:=t_{i}(q_{1,k}^{(j)}, \cdots, q_{n,k}^{(j)})\\&=t_{i}(x_{1}, \cdots, x_{j}-k, x_{j+1}+\overline{q_{j+1,k}}^{(j)}(x_{1}, \cdots, x_{j}), \cdots, x_{n}+\overline{q_{n,k}}^{(j)}(x_{1}, \cdots, x_{n-1})) \\&=x_{i}+\overline{q_{i,k}}^{(j)}(x_{1}, \cdots, x_{i-1}) \\&=t_{i}+\overline{q_{i,k}}^{(j)}(x_{1}, \cdots, x_{i-1}) \end{align*}

\end{enumerate}
\end{small}

By utilizing this information on the action of $a_{j}^{k}$ on each $t_{i}$ we are able to determine the complexity of the algorithm presented in Figure 1 when the coordinate functions are used.  We begin by finding an upper bound on the size of the matrices produced.  It is clear from the algorithm presented in \cite{nickel} that the size of the matrix representation produced depends on the size of the $\mathbb{Q}$-basis constructed using the algorithm presented in Figure 1.  To determine a bound, we begin by analyzing the main loop of the algorithm in Figure 1.  The number of iterations the main loop in Figure 1 undergoes throughout the algorithm is directly related to the size of the basis produced.

\begin{theorem}\label{dimthm} When the algorithm presented in \cite{nickel} is implemented using the coordinate functions $t_{i}$ for $1\leq i \leq n$ the worst case dimension of the matrix representation produced depends quadratically on the Hirsch length of the group; that is, the dimension of the matrix representation is $O(n^{2})$ where $n$ is the Hirsch length of the group.. \end{theorem}
\begin{proof}

We will prove the cardinality of the basis depends quadratically on the Hirsch length of the group by using induction on the number of times the main loop in Figure 1 is repeated.

Recall that $t_{i}^{a_{j}^{k}}:=t_{i}(q_{1,k}^{(j)}, \cdots q_{n,k}^{(j)})$ where $q_{i,k}^{(j)}=x_{i}$ for $1\leq i <j$, $q_{j,k}^{(j)}=x_{j}-k$, and $q_{i,k}^{(j)}=x_{i}+\overline{q_{i,k}}^{(j)}(x_{1}, \cdots, x_{i-1})$ for $j< i \leq n$. \\
First note that we may write  $\overline{q_{i,k}}^{(j)}=r_{i,k}^{(j)}+ \displaystyle \sum_{m=1}^{i-1}c_{m}t_{m}$ for appropriate $c_{m}\in \mathbb{Q}$, $r_{i,k}^{(j)}\in \mathbb{Q}[x_{1}, \cdots, x_{i-1}]$, where $r_{i,k}^{(j)}=0$ or cannot be written as a linear combination of the coordinate functions $t_{l}$  for $1\leq l \leq n$.\\

The algorithm presented in \cite{nickel} (see Figure 1) begins by building a basis from the coordinate functions $t_{1}, \cdots, t_{n}$ using \textit{Insert}.  By definition $t_{i}(x_{1}, \cdots, x_{n})=x_{i}$; hence, \textit{Insert} will form the basis $B=\left\{t_{1}, \cdots, t_{n}\right\}$. \\

Since we are considering the worst case scenario, we will assume that each polynomial $r_{i,k}^{(j)}$ is not a linear combination of $t_{1}, \cdots, t_{n}, r_{m,k}^{(j)} \text{ for } m\neq i$. \\

Claim: The basis produced in the worst case scenario is\\
\resizebox{\linewidth}{!}{$B=\left\{t_{1}, \cdots, t_{n}, r_{n,k_{n,n-1}}^{(n-1)}, \cdots,r_{n,k_{n,1}}^{(1)}, r_{n-1,k_{n-1,n-2}}^{(n-2)}, \cdots, r_{n-1,k_{n-1,1}}^{(1)}, \cdots, r_{3,k_{3,2}}^{(2)}, r_{3,k_{3,1}}^{(1)}, r_{2,k_{2,1}}^{(1)}, r_{1}\right\}$}\\ 
for $k_{i,j}=1, \cdots, m_{i,j}$ where $m_{i,j}$ denotes the number of terms in the polynomial $\overline{q_{i}}^{(j)}.$

We must first consider when the main loop in Figure 1 is iterated one time.   \\

 The loop begins by adding images of each $t_{i}$ under the action of $a_{n}$.  We may assume that the algorithm begins by adding images of $t_{n}$ under the action of $a_{n}$ to the basis.  We know that $t_{n}^{a_{n}}=t_{n}-1$, which clearly is not in the span of the basis.  \textit{Insert}($B$, $t_{n}^{a_{n}}$) will begin by subtracting $t_{n}$ from $t_{n}-1$, leaving $r=-1$ which will be added to the basis.  As it is necessary to close the $G$-module under the action of powers of $a_{n}$, we must repeat the process for $t_{n}^{a_{n}}$. Hence, the algorithm computes $(t_{n}^{a_{n}})^{a_{n}}=t_{n}^{a_{n}^{2}}=t_{n}-2$, which is the span of the basis; that is $r=0$ and the process terminates.   For $1\leq i \leq n-1$, we have that $t_{i}^{a_{n}}=t_{i}$. Hence, for $t_{i}$ with $1 \leq i \leq n-1$ no new polynomials will be added to the basis.  After this process is completed for $a_{n}$ the resulting basis will be $B=\left\{t_{1}, \cdots, t_{n}, r_{1}=-1 \right\}$.\\

\textit{Base case}: The main loop in Figure 1  is repeated twice.\\
The second repetition of the loop repeats the process for $a_{n-1}$. \\
 We begin by noting that for $1 \leq i \leq n-2$ we have $t_{i}^{a_{n-1}}=t_{i}$ and   $t_{n-1}^{a_{n-1}}=x_{n-1}-1=t_{n-1}+r_{1}$, which are in the span of the basis.
Hence, we need only to consider the case in which images of $t_{n}$ under the action of $a_{n-1}^{k}$ for $k\in \mathbb{N}$ are added to the basis. By the nature of the coordinate functions and since the polynomials $\overline{q_{n,k}}^{(n-1)}$ and $\overline{q_{n,j}}^{(n-1)}$ differ only by coefficients to close under the action of powers of $a_{n-1}$ the inner most loop will be repeated at most $m_{n,n-1}$ times, where $m_{n,n-1}$ denotes the number of terms in the polynomial $\overline{q_{n}}^{(n-1)}$; that is, we need to add images of $t_{n}$ under the action of $a_{n-1}^{k}$ for $k=1, \cdots, m_{n,n-1}$. 

From above we know that
\begin{align*}t_{n}^{a_{n-1}^{k}}=&t_{n}(x_{1}, \cdots, x_{n-1}-k, x_{n}+\overline{q_{n,k}}^{(n-1)}(x_{1}, \cdots, x_{n-1}))\\=&x_{n}+ \overline{q_{n,k}}^{(n-1)}(x_{1}, \cdots, x_{n-1})\\=& t_{n} + \overline{q_{n,k}}^{(n-1)}\\=& t_{n}+r_{n,k}^{(n-1)}+\displaystyle \sum_{m=1}^{n-1}c_{m}t_{m}. \end{align*}
 \textit{Insert}($B,t_{n}^{a_{n-1}}$) will subtract appropriate multiples of $t_{1}, \cdots, t_{n}$ from $t_{n}^{a_{n-1}^{k}}$ leaving $r=r_{n,k}^{(n-1)}$ to be added to the basis for $k=1, \cdots, m_{n,n-1}$. 
 Hence, after two iterations of the loop the resulting basis is \\$B=\left\{t_{1}, \cdots, t_{n},r_{n,k_{n,n-1}}^{(n-1)}, r_{1}\right\}$ for $k_{n,n-1}=1, \cdots, m_{n,n-1}$. \\

\textit{Inductive assumption:} Note that the $j^{th}$ iteration of the loop will add images of $t_{1}, \cdots, t_{n}$ under the action of powers of $a_{n-j+1}$.  Suppose for $2<j<n$ iterations of the loop the resulting basis is\\
\resizebox{\linewidth}{!}{$B=\left\{t_{1}, \cdots, t_{n}, r_{n,k_{n,n-1}}^{(n-1)}, \cdots, r_{n,k_{n,n-j+1}} ^{(n-j+1)}, r_{n-1,k_{n-1,n-2}}^{(n-2)}, \cdots, r_{n-1,k_{n-1,n-j+1}}^{(n-j+1)}, \cdots, r_{n-j+2,k_{n-j+2,n-j+1}}^{(n-j+1)},  r_{1},  \right\} $}
for $k_{h,i}=1, \cdots, m_{h,i}$ where each $m_{h,i}$ represents the number of terms in the polynomial $\overline{q_{h}}^{(i)}$. \\

\textit{$(j+1)^{st}$ iteration of the loop}: The main loop in Figure 1 will be performed for $a_{n-j}$.\\
We begin by noting that
\begin{tiny}
\begin{align*}
t_{i}^{a_{n-j}^{k}}:&=t_{i}(q_{1,k}^{(n-j)}, \cdots, q_{n,k}^{(n-j)})\\=&t_{i}(x_{1},\cdots, x_{n-j}-k, x_{n-j+1}+r_{n-j+1,k}^{(n-j)}+ \displaystyle \sum_{m=1}^{n-j}c_{m}t_{m}, \cdots, x_{n}+r_{n,k}^{(n-j)}+\displaystyle \sum_{m=1}^{n-1}d_{m}t_{m})
\end{align*}
\end{tiny}

We would like to see which polynomials will be added to the basis in this step.  We will look at three cases:
\begin{itemize}
\item  $1\leq i \leq n-j-1$:\\
$t_{i}^{a_{n-j}^{k}}=x_{i}=t_{i}.$ This is in the span of the basis, and so the algorithm will have $r=0$ and the process terminates.
\item $i=n-k$:\\
$t_{n-j}^{a_{n-j}^{k}}=x_{n-j}-k=t_{n-j}+kr_{1}$. This is in the span of the basis, and so the algorithm will have $r=0$ and the loop terminates.
\item $n-j+1\leq i \leq n$:\\
$t_{i}^{a_{n-j}^{k}}=x_{i}+r_{i,k}^{(n-j)}+\displaystyle \sum_{m=1}^{i-1}c_{m}t_{m}$.  By assumption $r_{i,k}^{(n-j)}$ is not in the span of the basis for each $n-j+1 \leq i \leq n$ and $k\in \mathbb{N}$. The algorithm will then add $r_{i,k_{i,n-j}}^{(n-j)}$ for $k_{i,n-j}=1, \cdots, m_{i,n-j}$ to the basis for each $i$, resulting in the basis \\
\resizebox{\linewidth}{!}{$B=\left\{t_{1}, \cdots, t_{n}, r_{n,k_{n,n-1}}^{(n-1)}, \cdots, r_{n,k_{n,n-j}} ^{(n-j)}, r_{n-1,k_{n-1,n-2}}^{(n-2)}, \cdots, r_{n-1,k_{n-1,n-j}}^{(n-j)}, \cdots, r_{n-j+2,k_{n-j+2,n-j}}^{(n-j)}, r_{n-j+1,k_{n-j+1,n-j}}^{(n-j)},  r_{1} \right\}$}\\ 
for $k_{h,i}=1,\cdots, m_{h,i}$.

\end{itemize}

Hence, for any $1\leq j<n$ we know after $j+1$ iterations the resulting basis will be \\
\resizebox{\linewidth}{!}{$B=\left\{t_{1}, \cdots, t_{n}, r_{n,k_{n,n-1}}^{(n-1)}, \cdots, r_{n,k_{n,n-j}} ^{(n-j)}, r_{n-1,k_{n-1,n-2}}^{(n-2)}, \cdots, r_{n-1,k_{n-1,n-j}}^{(n-j)}, \cdots, r_{n-j+2,k_{n-j+2,n-j}}^{(n-j)}, r_{n-j+1,k_{n-j+1,n-j}}^{(n-j)},  r_{1} \right\}$}\\ 
for $k_{h,i}=1,\cdots, m_{h,i}$ where each $m_{h,i}$ represents the number of terms in the polynomial $\overline{q_{h}}^{(i)}$.
 The main loop in Figure 1 will be repeated $n$ times (one for each element in the nilpotent generating sequence), and the resulting basis will be\\
\resizebox{\linewidth}{!}{ $B=\left\{t_{1}, \cdots, t_{n}, r_{n,k_{n,n-1}}^{(n-1)}, \cdots,r_{n,k_{n,1}}^{(1)}, r_{n-1,k_{n-1,n-2}}^{(n-2)}, \cdots, r_{n-1,k_{n-1,1}}^{(1)}, \cdots, r_{3,k_{3,2}}^{(2)}, r_{3,k_{3,1}}^{(1)}, r_{2,k_{2,1}}^{(1)}, r_{1}\right\}$}\\ 
for $k_{i,j}=1, \cdots, m_{i,j}$ as desired.

 The cardinality of this basis is at most\\ $n+\displaystyle\sum_{i=1}^{n-1} m_{n,i}+\displaystyle\sum_{i=1}^{n-2} m_{n-1,i}+\cdots +\displaystyle\sum_{i=1}^{2} m_{3,i}+m_{2,1}+1$. 

Let $m$ denote the largest of the $m_{i,j}$.  Then the cardinality of the basis is bounded above by
\begin{align*}n+(n-1)m+(n-2)m+\cdots +2m+m+1\leq& m\displaystyle \sum_{i=1}^{n} i +1\\ \leq&\dfrac{m}{2}n(n+1)+1.\end{align*}
Thus, the dimension of the matrix representation is $O(n^{2})$.

\end{proof}

\begin{theorem}\label{comthm} The running time of the algorithm presented in Figure 1 using the coordinate functions $t_{i}$ for $1 \leq i \leq n$ is $O(n^{l+2})$, where $n$ is the Hirsch length of the group and $l$ is the highest degree of the polynomials $t_{i}^{a_{j}}$.
\end{theorem}

\begin{proof}
Before determining the worst case complexity of the algorithm presented in Figure 1, we must determine the number of steps needed by the method \textit{Insert}.  \textit{Insert} takes a basis of polynomials in ascending order and a polynomial $f$, and determines if $f$ is in the span of the basis by subtracting the appropriate multiple of each polynomial in the basis, starting with the polynomial that is largest with respect to the ordering and continuing in descending order. If $f\neq 0$ after this process is completed, then $f$ is not in the span of the basis and is added to the original basis.  It is clear that the number of subtractions used on a given polynomial $f$ by \textit{Insert} is at most the number of terms in $f$.  Each polynomial $t_{i}^{a_{j}}$ is a polynomial in $n$ variables since $t_{i}^{a_{j}}=t_{i}(q_{1}^{(j)}, \cdots, q_{n}^{(j)})$.  An $l^{th}$ degree polynomial with $n$ variables has at most $\displaystyle \sum_{k=1}^{l}$ $k+n-1 \choose k$ terms, which is bounded above by $\dfrac{c}{l!}(n+l-1)^{l}$ for an appropriate $c\in \mathbb{N}$. We will denote the number of terms in a given polynomial $f$ by $m_{f}$.\\

Recall the algorithm to form a $\mathbb{Q}$-basis of the $G$-module generated by the coordinate functions $t_{i}$ for $1 \leq i \leq n$ begins by building a basis from the coordinate functions $t_{1}, \cdots, t_{n}$ by implementing \textit{Insert}($B$, $t_{i}$) for $1 \leq i \leq n$. It is clear from the above remarks that this process will take a total of $\displaystyle \sum_{i=1}^{n} m_{t_{i}}$ steps. 

The next step to determine the complexity of the algorithm for building a $\mathbb{Q}$-basis is to determine how many iterations of the innermost loop in Figure 1 occur for each $a_{j}$.  In the proof of Theorem \ref{dimthm} one can see how many iterations of Insert need to be performed for each $a_{j}$, where $1\leq j \leq n$. It is clear from the proof of Theorem \ref{dimthm} that when the first iteration of the loop is performed (closing under the action of $a_{n}$) \textit{Insert} is repeated $n+1$ times.  For the $k+1^{st}$ iteration the algorithm will close the module under the action of $a_{n-k}$.  For $t_{i}$, with $1\leq i \leq n-k$, $t_{i}^{a_{n-k}}$ is already in the basis and hence $n-k$ iterations of \textit{Insert} are performed. For $t_{i}$ with $n-k+1 \leq i \leq n$, $t_{i}^{a_{n-k}}$ is not in the span of the basis and \textit{Insert} is performed $m_{i,n-k}$ times, where $m_{i,n-k}$ denotes the number of terms in the polynomial $\overline{q_{i}}^{(n-k)}$, for each $i$ (see proof of Theorem \ref{dimthm}).  Hence, here \textit{Insert} will be performed $\displaystyle \sum_{i=n-k+1}^{n}m_{i,n-k}$ times. Thus, the total number of times Insert will be performed in the $k+1^{st}$ iteration of the loop is $(n-k)+\displaystyle \sum_{i=n-k+1}^{n}m_{i,n-k}$. As above let $m$ denote the maximum of the $m_{i,j}$.

The total number of iterations of Insert performed in the main loop of Figure 1 is then
 \begin{align*}(n+1)+\displaystyle \sum_{k=1}^{n-1}(n-k)+\displaystyle \sum_{k=1}^{n-1}\sum_{i=n-k+1}^{n}m_{i,n-k}&=1+\displaystyle\sum_{k=1}^{n}k+\displaystyle \sum_{k=1}^{n-1}\sum_{i=n-k+1}^{n}m_{i,n-k}\\=& 1+\dfrac{1}{2}n(n+1)+\displaystyle \sum_{k=1}^{n-1}\sum_{i=n-k+1}^{n}m_{i,n-k}\\ \leq&1+\dfrac{1}{2}n(n+1)+\displaystyle \sum_{k=1}^{n-1}km\\=&1+\dfrac{1}{2}n(n+1)+\dfrac{m}{2}n(n-1)\\=&1+\dfrac{m+1}{2}n^{2}+\dfrac{1-m}{2}n\end{align*}

Let $l$ be the highest degree of the polynomials utilized in the algorithm.  Then we know the maximum number of steps that \textit{Insert} requires at any step is bounded above by $\dfrac{c}{l!}(n+l-1)^{l}$. Then the number of steps required by the main loop in Figure 1  is bounded above by 

\begin{eqnarray*}\dfrac{c}{l!}(n+l-1)^{l}(1+\dfrac{m+1}{2}n^{2}+\dfrac{1-m}{2}n)\end{eqnarray*} and the total number of steps required by the algorithm is bounded above by 
\resizebox{\linewidth}{!}{$\displaystyle \sum_{i=1}^{n} m_{t_{i}}+\dfrac{c}{l!}(n+l-1)^{l}(1+\dfrac{m+1}{2}n^{2}+\dfrac{1-m}{2}n) \leq \dfrac{c}{l!}(n+l-1)^{l}(1+\dfrac{m+1}{2}n^{2}+\dfrac{3-m}{2}n).$} Thus, the total running time of the algorithm is $O(n^{l+2})$.

\end{proof}

\section{Modified Algorithm}
From Theorem \ref{dimthm}, we see that the $\mathbb{Q}$-basis of the $G$-module generated by the coordinate functions $t_{1}, \cdots, t_{n}$ is a subset of \\
\resizebox{\linewidth}{!}{$B=\left\{t_{1}, \cdots, t_{n}, r_{n,k_{n,n-1}}^{(n-1)}, \cdots,r_{n,k_{n,1}}^{(1)}, r_{n-1,k_{n-1,n-2}}^{(n-2)}, \cdots, r_{n-1,k_{n-1,1}}^{(1)}, \cdots,\\ r_{3,k_{3,2}}^{(2)}, r_{3,k_{3,1}}^{(1)}, r_{2,k_{2,1}}^{(1)}, r_{1}\right\}$}
\\ 
for $k_{i,j}=1, \cdots, m_{i,j}$ where $m_{i,j}$ denotes the number of terms in the polynomial $\overline{q_{i}}^{(j)}.$
 Using this fact we may modify the algorithm presented by Nickel in \cite{nickel}.   \\

 In order to find a $\mathbb{Q}$-basis of the $G$-module generated by the coordinate functions $t_{1}, \cdots, t_{n}$ we must begin by building a basis up from the coordinate functions using \textit{Insert} as is done in \cite{nickel}.
 By the nature of the coordinate functions, it is clear that the action of each $a_{j}$ for $1\leq j \leq  n$ will only add new polynomials to the basis at certain steps in the algorithm.  Recall that in the first time the main loop in Figure 1 is repeated the only polynomial that could be added to the basis is when $Insert(B, t_{n}^{a_{n}})$ is applied.  Then in the $k+1^{st}$, where $1\leq k \leq n-1$,  iteration of the loop the only polynomials that could be added to the basis arise from closing under powers of $a_{n-k}$ for $n-k+1 \leq i \leq n$; that is, we need only perform $Insert(B, t_{n}^{a_{n}})$ and $Insert(B,t_{i}^{a_{j}^{p}})$ for $i>j$, $1\leq j \leq n-1$ and for $p\in \mathbb{N}$. This is clear in the proof of Theorem \ref{dimthm}.   Hence, in order to determine the $\mathbb{Q}$-basis of the $G$-module generated by $t_{1},\cdots,t_{n}$ it is enough to implement $Insert$ for only these polynomials. The pseudo-code for the algorithm is displayed in Figure 2.\\
\begin{figure}[t]
\label{Building a Basis using Coordinate Functions}
\caption{Building a Basis for a $G$-module using Coordinate Functions}
\flushleft
\textbf{Input:} A finitely generated torsion-free nilpotent group $G$ with nilpotent generating sequence $a_{1}, \cdots,a_{n}$ and multiplication polynomials $q_{1}, \cdots, q_{n}$; coordinate functions $t_{1}, \cdots, t_{n}$.\\
\textbf{Output:} A $\mathbb{Q}$-basis $B$ for the $G$-module generated by $t_{1}, \cdots, t_{n}$.
\textbf{Algorithm:}\\
$B:=[ ]$;\\
 \flushleft \quad \textbf{for} $j$ \textbf{in} $[1, \cdots n]$ \textbf{do } $Insert(B,t_{j})$; \textbf{od};\\
\flushleft \quad \textbf{do} $Insert(B,t_{n}^{a_{n}})$; \textbf{od};\\
 \quad\textbf{for} $j$ \textbf{in} $[n-1\cdots 1]$ \textbf{do}\\
 \qquad \textbf{for} $i$ \textbf{in} $[1, \cdots n]$ \textbf{do}\\
\qquad\quad \textbf{while} $ i>j$ \textbf{do}\\
 \quad \quad \qquad \textbf{repeat} \\
 \quad \qquad \qquad  \# Compute $t_{i}^{a_{j}}$\\
 \quad \qquad \qquad $t_{i}^{a_{j}}:=t_{i}(q_{1}^{(j)}, \cdots, q_{n}^{(j)})$;\\
 \quad \qquad \qquad r:=Insert($B,t_{i}^{a_{j}}$);\\
 \quad \qquad \qquad $t_{i}:=t^{a_{j}}$;\\
 \quad \qquad \textbf{until} r=0;\\
\qquad \quad \textbf{od};\\
 \qquad \textbf{od};\\
 \quad \textbf{od};\\
 \flushleft \textbf{return} $B$;

 \end{figure}
\begin{theorem}
The algorithm presented in Figure 2 has running time $O(n^{l+2})$, where $n$ is the Hirsch length of the group and $l$ is the highest degree of the polynomials $t_{i}^{a_{j}}$.
\end{theorem}
\begin{proof}
We denote the number of terms in a polynomial $f$ by $m_{f}$.  Then the number of steps required for $Insert(B,t_{i})$ for $1\leq i \leq n$ is at most $m_{t_{i}}$ and the number of steps of $Insert(B,t_{n}^{a_{n}})$  is at most $m_{t_{n}}$.   

Next, we must determine the number of times $Insert$ is performed for the $k+1^{st}$ iteration of the innermost \textbf{for} loop in Figure 3; that is, we must determine the number of times $Insert$ is performed to close the module under the action of $a_{n-k}$. As the only difference between the modified algorithm and the algorithm presented in \cite{nickel} when implemented with the coordinate functions is the \textbf{while} loop, the number of times $Insert$ is performed for each $a_{n-k}$ only differs by $n-k$ (see proof of \ref{comthm}). Thus, the total number of times Insert will be performed in the $k+1^{st}$ iteration of the loop is $\displaystyle \sum_{i=n-k+1}^{n}m_{i,n-k}$. Let $m$ denote the maximum of the $m_{i,j}$.

The total number of iterations of Insert performed is then
 \begin{align*}\displaystyle \sum_{k=1}^{n-1}\sum_{i=n-k+1}^{n}m_{i,n-k}&\leq \displaystyle \sum_{k=1}^{n-1}km\\=&\dfrac{m}{2}n(n-1)\\=&\dfrac{m}{2}n^{2}-\dfrac{m}{2}n\end{align*}

Let $l$ be the highest degree of the polynomials utilized in the algorithm.  Then we know the maximum number of steps that \textit{Insert} requires at any step is bounded above by $\dfrac{c}{l!}(n+l-1)^{l}$ for some $c\in \mathbb{N}$. Then the number of steps required by the main loop in Figure 2  is bounded above by 

\begin{eqnarray*}\dfrac{c}{l!}(n+l-1)^{l}(\dfrac{m}{2}n^{2}-\dfrac{m}{2}n)\end{eqnarray*} and the total number of steps required by the algorithm is bounded above by \begin{small}
\begin{eqnarray*}\displaystyle \sum_{i=1}^{n} m_{t_{i}}+m_{t_{n}}+\dfrac{c}{l!}(n+l-1)^{l}(\dfrac{m}{2}n^{2}-\dfrac{m}{2}n) \leq \\ \dfrac{c}{l!}(n+l-1)^{l}(1+\dfrac{m}{2}n^{2}+\dfrac{2-m}{2}n).\end{eqnarray*} \end{small} Thus, the total running time of the algorithm is $O(n^{l+2})$.
\end{proof}

\section{Closing Remarks}
We have shown that the algorithm in \cite{nickel} has polynomial time complexity with respect to the Hirsch length if the coordinate functions are utilized, and that the dimension of the matrix representation produced depends quadratically on the Hirsch length of the group.  It is possible to improve the bound on the dimension and time complexity by using a different set of functions rather than the coordinate functions. Unfortunately, there is no known way to choose these functions.  \\

The modified algorithm is only for generating a $\mathbb{Q}$-basis
for the $G$-module generated by the coordinate functions unlike the
algorithm presented in \cite{nickel}, which works for a general set
of functions $f_{1}, \cdots,f_{k}$.  We would like to note that the
running time of the modified algorithm can have better running times than the algorithm presented in \cite{nickel} since the running time is dependent on the polynomials $t_{i}^{a_{j}}$. While this algorithm potentially decreases running time, it does not affect the dimension of the matrices produced.

\section*{Acknowledgement} The authors would like to thank the anonymous referee for his helpful comments. Delaram Kahrobaei is grateful to Professor
Derek Holt for providing the opportunity for her to visit Warwick
University in summer 2010 and very stimulating conversations
regarding this problem.
\bibliographystyle{plain}
\bibliography{XBib}

\begin{thebibliography}{1}

\bibitem{degraaf}
W.~A.~De Graaf and W.~Nickel.
\newblock Constructing faithful representations of finitely-generated
  torsion-free nilpotent groups.
\newblock {\em J. Symbolic Computation}, 33:31--41, 2002.

\bibitem{PH57}
Philip Hall.
\newblock Nilpotent groups.
\newblock {\em Notes of lectures given at the Canadian Mathematical Congress,
  University of Alberta}, pages 1--42, 1957.

\bibitem{holt}
D.~Holt, B.~Eick, and E.~O'Brien.
\newblock {\em Handbook of computational group theory}.
\newblock Discrete Mathematics and its Applications (Boca Raton), Chapman and
  Hall/CRC, Boca Raton, FL, 2005.

\bibitem{kos}
A.I. Kostrikin and I.R. Shafarevich.
\newblock {\em Encyclopedia of Mathematical Sciences, Volume 37, Algebra IV:
  Infinite Groups, Linear Groups}.
\newblock Springer-Verlag, 1991.

\bibitem{LGS}
C.R. Leedham-Green and L.H. Soicher.
\newblock Symbolic collection using deep thought.
\newblock {\em LMS J. Comput. Math.}, 1:9--24, 1998.

\bibitem{nickel}
W.~Nickel.
\newblock Matrix representations for torsion-free nilpotent groups by deep
  thought.
\newblock {\em Journal of Algebra}, 300:376--383, 2006.

\end{thebibliography}

\end{document}